\documentclass{amsart}

\usepackage{graphicx}
\usepackage[colorlinks=true, allcolors=blue]{hyperref}
\usepackage{cleveref}
\usepackage{booktabs}
\usepackage{mathrsfs}

\newcommand{\pplusac}[1]{\mathcal{P}_{+,ac}\left(#1\right)}
\newcommand{\ptwoac}[1]{\mathcal{P}_{2,+,ac}\left(#1\right)}
\newcommand{\p}[1]{\mathcal{P}\left(#1\right)}
\newcommand{\lorenz}[1]{\mathfrak{L}\left(#1\right)}
\newcommand{\lorenzsub}[1]{\mathfrak{L}_{#1}}
\newcommand{\real}{\mathbb{R}}
\newcommand{\posreal}{\mathbb{R}^+}

\newcommand{\ddy}[1]{\frac{\partial #1}{\partial y}}

\newcommand{\ddx}[1]{\frac{\partial #1}{\partial x}}
\newcommand{\ddt}[1]{\frac{\partial #1}{\partial t}}

\newcommand{\ddf}[1]{\frac{\partial #1}{\partial f}}
\newcommand{\ddg}[1]{\frac{\partial #1}{\partial g}}

\newcommand{\twoddx}[1]{\frac{\partial^2 #1}{\partial x^2}}
\newcommand{\twoddy}[1]{\frac{\partial^2 #1}{\partial y^2}}

\newcommand{\twoddf}[1]{\frac{\partial^2 #1}{\partial f^2}}
\newcommand{\twoddg}[1]{\frac{\partial^2 #1}{\partial g^2}}

\newcommand{\dd}[2]{\frac{\partial #1}{\partial #2}}

\newcommand{\totdd}[2]{\frac{d #1}{d #2}}
\newcommand{\tottwodd}[2]{\frac{d^2 #1}{d {#2}^2}}

\newcommand{\ddNorm}[1]{\frac{\partial #1}{\partial \nu}}

\newcommand{\frechet}[2]{\frac{\delta #1}{\delta #2}}

\newcommand{\divg}{\text{div}}
\newcommand{\grad}[2]{\text{grad}_{#1}#2}

\newcommand{\cdspace}{{\mathcal{C}_D}}
\newcommand{\cdspacearg}[1]{{\mathcal{C}_{#1}}}

\newcommand{\LL}{\mathcal{L}}
\newcommand{\KK}{\mathbb{K}}
\newcommand{\LLfrak}{\mathfrak{L}}
\newcommand{\Lbf}{\mathbf{L}}
\newcommand{\Pbf}{\mathbf{P}}
\newcommand{\LorTan}{\mathbb{T}}

\newcommand*{\defeq}{\mathrel{\vcenter{\baselineskip0.5ex \lineskiplimit0pt
                     \hbox{\scriptsize.}\hbox{\scriptsize.}}}%
                     =}

\theoremstyle{plain}
\newtheorem{theorem}{Theorem}[section]
\newtheorem{lemma}[theorem]{Lemma}

\newtheorem{proposition}[theorem]{Proposition}

\theoremstyle{definition}
\newtheorem{definition}[theorem]{Definition}
\newtheorem{example}[theorem]{Example}

\newtheorem*{principle*}{Principle}
\newtheorem*{question*}{Question}

\theoremstyle{remark}
\newtheorem{remark}[theorem]{Remark}

\numberwithin{equation}{section}

\title[Gradient flows of Lorenz curves]{Variational principles on the space of Lorenz curves: \\ {\small Gradient structures and isometries inspired by Wasserstein geometry}}

\author{David W. Cohen}
\address{Department of Mathematics, Tufts University, Medford, MA}
\email{David.Cohen@Tufts.edu}
\thanks{The author received support from the NDSEG fellowship.}



\date{July 2025}

\begin{document}

\begin{abstract}

We motivate and derive novel Riemannian gradient structures on the space of Lorenz curves, which preserve infinite-dimensional variational principles inherited from Fokker-Planck equations via the lens of Wasserstein geometry and its variants. We also prove isometry results between corresponding formal manifolds of probability measures and Lorenz curves, which suggest meaningful metrics on the space of Lorenz curves when an underlying kinetic premise is present. In so doing, elegant variational principles are imbued upon highly nonlinear and nonlocal integro-differential evolution equations resulting from a recently derived variable transformation of McKean-Vlasov Fokker-Planck equations.

\end{abstract}

\maketitle

\section{Introduction}

For a positive probability density $\rho$ over $\posreal$, which represents the distribution of a population among possible wealths, the \textit{Lorenz curve}, a frequently studied object in the analysis of economic inequality, is the parametric curve $\left(\int_0^x\,dy\,\rho(y), \int_0^x\,dy\, \rho(y)y\right)$. Such representations were first described by the economist Max Lorenz in the early twentieth century.\footnote{We refer readers to \cite{melot2024,MR3012052} as useful references on Lorenz curves.}

By the monotonicity of the cumulative distribution function (cdf) for a positive probability density, the Lorenz curve $\LL(f)$ can be considered a function of the cdf, $f\in [0,1]$, as a new independent variable. \Cref{fig:timeVaryingDistAndLor} shows a sequence of positive distributions over $\posreal$ and their associated Lorenz curves.

Thus for an equation of motion for a wealth distribution, there is an intrinsic time-dependent evolution equation of the associated Lorenz curves. Such a transform was derived for the first time in \cite{LorDyn2024}.

Consider the McKean-Vlasov Fokker-Planck equation \begin{equation}\label{eq:MVFPE}
    \ddt{\rho(x,t)} = -\ddx{}\left[\Sigma[x,t,\rho(x,t)] \rho(x,t)\right] + \twoddx{}\left[D[x,t,\rho(x,t)] \rho(x,t)\right],
\end{equation} where $\Sigma: \real \times \posreal \times \pplusac{\real} \rightarrow \real$ is the drift coefficient and $D: \real \times \posreal \times \pplusac{\real} \rightarrow \posreal$ is the diffusion coefficient. It was shown that the transformed dynamics of the Lorenz curve of $\rho(x,t)$, denoted $\LL: [0,1] \times \posreal \rightarrow \real$, obey \begin{equation}\label{eq:LorDynGen}
     \ddt{\LL(f,t)} = - \frac{\widetilde{D}[f,t,\LL]}{\twoddf{\LL(f,t)}} + \int_0^f\,dg\,\widetilde{\Sigma}[g,t,\LL],
\end{equation} where \[\widetilde{\Sigma}[f,t,\LL] \defeq \Sigma\left[\ddf{\LL(f,t)},t,\left(\twoddf{\LL(f,t)}\right)^{-1}\right] \]
 and \[\widetilde{D}[f,t,\LL] \defeq D\left[\ddf{\LL(f,t)},t,\left(\twoddf{\LL(f,t)}\right)^{-1}\right].\]

\begin{figure}[ht]
\centering
\includegraphics[width=.55\textwidth]{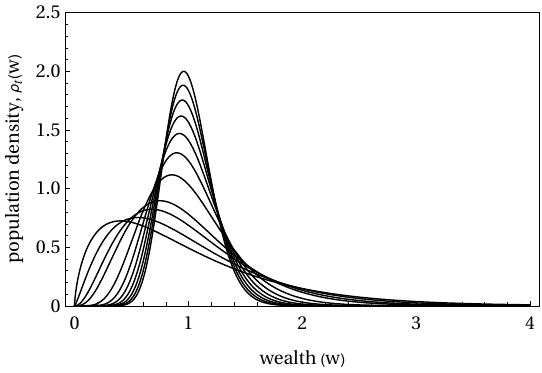}\\
\includegraphics[width=.55\textwidth]{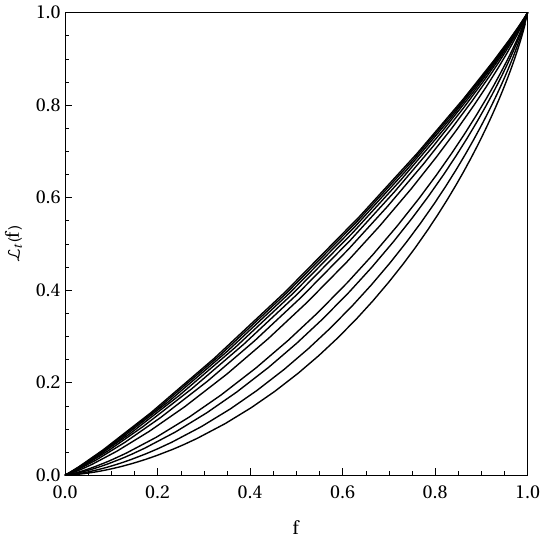}
\caption{A time-varying positive distribution with unit first moment over the positive reals (top) and the associated evolution of its Lorenz curves (bottom).}
\label{fig:timeVaryingDistAndLor}
\end{figure}

The equations governing the evolution in the space of Lorenz curves are nonlinear, nonlocal, and integro-differential even when the associated dynamics in the space of probability measures are linear, local, and only differential. It then appears folly to transform an evolutionary system into a facially worse set of variables. Below we set out the motivation in spite of this appearance.

\subsection{Motivation for pursuing the transformation}
The application of techniques from kinetic theory and mean-field theory to idealized, stochastic economic models has yielded exciting results over the past three decades -- for example, see \cite{MR4544046,MR2551376,MR3872473} and the citations within. It has become known that a broad class of econophysics models exhibit dynamics that asymptotically limit to the state of total oligarchy \cite{MR4267576,CDSpaces2024}. That is, a state in which almost all of the probability mass is concentrated ever closer to zero wealth and a vanishingly small amount of probability mass (colloquially, the oligarch) is sent off to unimaginably large wealth.

These solutions look roughly like \begin{equation*}\label{eq:runawayWealthSol}
    \rho_\epsilon(w)= \frac{1}{1+\epsilon}\delta(w-\epsilon) + \frac{\epsilon}{1+\epsilon} \delta\left(w-\frac{1}{\epsilon}\right)
\end{equation*} as $\epsilon  \rightarrow 0$.

A primary difficulty of this critical phenomenon is that the oligarch object \begin{equation}\label{eq:oligarchDist}``\Xi(w) = \lim _{\epsilon \rightarrow 0}\frac{\epsilon}{1+\epsilon}\delta \left(w-\frac{1}{\epsilon}\right)"\end{equation} does not fit cleanly within the classical theory of weak solution of PDE via Sobolev-Schwartz theory. This hinders the analysis of the evolution equations for the distributions of wealth.

Studying the evolution of the Lorenz curve, rather than the population distribution over wealth space $\posreal$, may alleviate these difficulties; since, in the Lorenz curve framework a partial or total oligarchy becomes a simple jump discontinuity at $f=1$. Therefore there is promise for understanding the solutions to the evolution equations of the Lorenz curves using the classical theory of weak solutions.

Therefore, while the evolution equations of \cref{eq:LorDynGen} do not appear \textit{prima facie} to be more pleasant, our motivation is to avoid dealing with the ``oligarch distribution'' of \cref{eq:oligarchDist} represented as population density in wealth space.

In spite of the Lorenz equations' apparent nastiness, our main results provide novel differential geometric understandings of the space of Lorenz curves that imbue infinite-dimensional variational principles to the $\LL$-dynamics associated to common gradient flows in the space of probability measures.

Namely, we show that formally viewing the space of Lorenz curves as an infinite-dimensional Riemannian manifold preserves gradient structures arising from $2$-Wasserstein theory and its variants. The variational structure of Fokker-Planck equations is not lost in the intrinsic Lorenz dynamics, and a curve of steepest ascent in the space of probability measures remains so once transformed to the space of Lorenz curves when viewed using a ``compatible'' geometry, considered here for the first time.

\begin{figure}[ht]
\centering
\includegraphics[trim={7cm 7cm 4cm 4.5cm},clip,width=.75\textwidth]{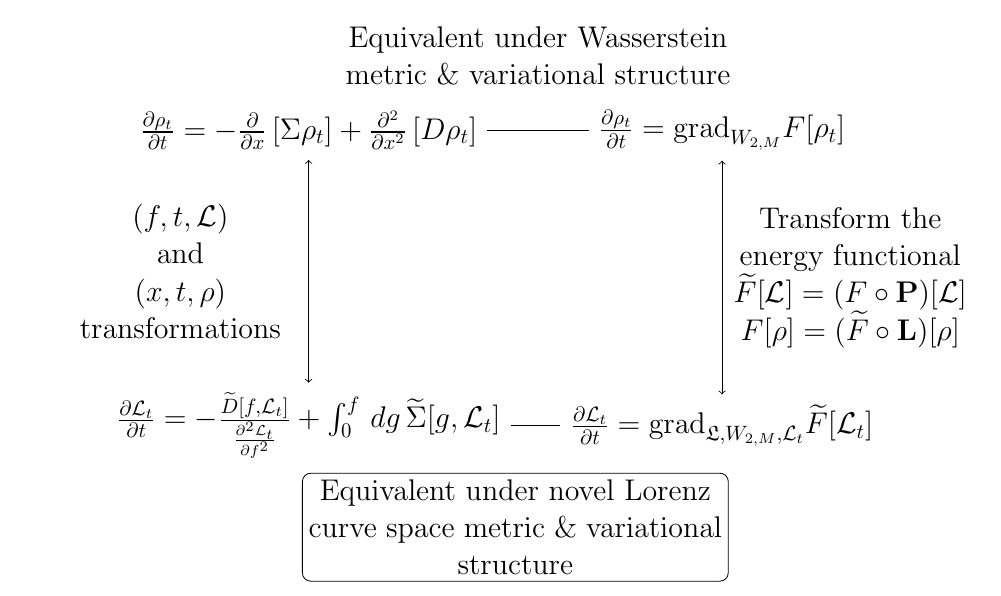}
\caption{The left hand side of this diagram was established in \cite{LorDyn2024}; whereas, the top portion has been actively explored since around 2000 \cite{MR2565840,MR1617171,MR1842429}. The portion of this diagram focused on in the paper is the boxed text.}
\label{fig:mainPoint1}
\end{figure}

\subsection{Outline}
In \cref{sec:prelims}, we set out the notation used in the paper, state some technical results that are used across the different sections, and codify the notion of the tangent space of Lorenz curves. Some key results of \cite{LorDyn2024}, which are repeatedly used in the proofs below, are stated as a convenient reference for the reader in \cref{sec:reviewOfResultsLorDyn}.

\Cref{sec:WassersteinGradFlowCompat,sec:w2MIsom} present the results regarding Wasserstein theory. We endow the space of Lorenz curves with a novel infinite-dimensional Riemannian geometry such that the variational principles of classical (\cref{ssec:linMobW2}) and non-linear (\cref{ssec:nonlinMobW2M}) mobility 2-Wasserstein gradient flows are preserved in the transformed dynamics. Highly nonlinear equations in the space of Lorenz curves are shown to evolve as curves of steepest descent. This relationship is depicted in \cref{fig:mainPoint1}. 

After establishing the variational principles in both sets of variables, an isometry result is proven in \cref{sec:w2MIsom} using the Riemannian metric structures identified in the prior section.

In \cref{sec:cdSpaceTheory}, similar results are presented for a formal infinite-dimensional Riemannian manifold over the space of probability measures that was first investigated in \cite{CDSpaces2024} and closely relates to evolution equations arising from econophysics models. In particular, the metric tensor of the $\cdspace$ spaces is different from that of $W_{2,M}$ and thus the compatible Riemannian structures differ as well. 

The novel metrics over the space of Lorenz curves offer physically meaningful ways to measure distances when an underlying kinetic premise is assumed, which is discussed at the end of \cref{sec:cdSpaceTheory}.

The Riemannian structures introduced in the paper are summarized, by their Onsager operators, in \cref{tab:onsagerOpComparison}.

\begin{table}[ht]
    \centering
    \begin{tabular}{ll|cc} 
        \toprule
        \multicolumn{2}{c|}{\textbf{Probability measures}} & \multicolumn{2}{c}{\textbf{Lorenz curves}} \\
        \midrule \midrule
        \textbf{Space} & \textbf{Onsager operator} & \textbf{Compatible space} & \textbf{Onsager operator} \\
        \midrule
        $W_2$ & $-\partial_x(\rho \partial_x [\cdot])$ & $\lorenzsub{W_2}$ & $-(\partial_{ff})^{-1}[\cdot]$ \\
        $W_{2,M}$ & $-\partial_x(M(\rho)\partial_x[\cdot])$ & $\lorenzsub{W_{2,M}}$ & $-\int^fdg \,m(1/\LL_{gg})\partial_g (\partial_{gg})^{-1}[\cdot]$ \\
        $\cdspacearg{\rho}$ & $\partial_{xx}(\rho\partial_{xx} [\cdot])$ & $\lorenzsub{\cdspacearg{\rho}}$ & $(\LL_{ff})^{-2}\times[\cdot]$ \\
        $\cdspace$ & $\partial_{xx}(D[x,\rho]\partial_{xx}[\cdot])$ & $\lorenzsub{\cdspace}$ & $\widetilde{d}[f,\LL](\LL_{ff})^{-2}\times[\cdot]$ \\
        \bottomrule
    \end{tabular}
    \caption{A comparison of the Onsager operators for the compatible manifolds presented here. The gradient structure is given by the Onsager operator acting on the differential of a functional.}
    \label{tab:onsagerOpComparison}
\end{table}

\section{Preliminaries}\label{sec:prelims}
\subsection{Notation}\label{ssec:notation}
The space of positive probability distributions defined over a set $\Omega \subset \mathbb{R}^n$ that are absolutely continuous with respect to the Lebesgue measure will be denoted by $\pplusac{\Omega}.$ The subset of these with finite second moment is $\ptwoac{\Omega}$.

The Lorenz curve of $\rho \in \pplusac{\real}$ is $\Lbf[\rho]$. Therefore $\Lbf[\rho]$ is a map from $[0,1]$ to $\mathbb{R}$. Given a Lorenz curve, $\LL$, the positive probability density $\Pbf[\LL]\in\pplusac{\real}$ satisfies $\left(\Lbf\circ\Pbf\right)[\LL] = \LL$. The cumulative distribution function (cdf) of a density $\rho$ will be denoted by $C[\rho]$ and the inverse cdf by $G[\rho].$

Where $I\subseteq \mathbb{R}$ is an interval, $\lorenz{I} = \left\{ \Lbf[\rho] \, :\, \rho \in \pplusac{I}\right\}$. That is, $\lorenz{I}$ is the set of Lorenz curves over the base space $I$, an interval of $\mathbb{R}$.

When there is a formal infinite-dimensional Riemannian manifold structure over some subset of $\pplusac{I}$, such as $W_2(I)$, and we are defining a \textit{compatible}\footnote{This notion is made clear later.} metric structure over the set $\lorenz{I}$, we will refer to resulting infinite-dimensional Riemannian manifold of Lorenz curves by, for example, $\lorenzsub{W_2}$. Often the subset $I$  of $\real$ will be omitted for notational clarity. The four considered here are: $\lorenzsub{W_2}$, $\lorenzsub{W_{2,M}}$, $\lorenzsub{\cdspacearg{\rho}}$, and $\lorenzsub{\cdspace}$.

Let $\rho(x,t): \Omega \times \mathbb{R} \rightarrow \mathbb{R}^+$ be such that for each $t$, $\rho(\cdot, t) \in \pplusac{\Omega}.$ That is, $\rho(\cdot,t)$ is a time-parameterized curve through the space of positive probability distributions. Whenever there is a time-parameterized curve through a space of functions, we let $\rho_t(x) = \rho(x,t).$ \textit{The subscript $t$ is not a time derivative.}

The shorthand for a time derivative will be a centered dot above a time-dependent quantity, such as $\dot{\rho}_t(x).$

For a functional $F: \pplusac{I} \rightarrow \mathbb{R}$, we always use $\widetilde{F}:\lorenz{I} \rightarrow \mathbb{R}$ to denote the functional such that \[F[\rho] = \left(\widetilde{F}\circ\Lbf\right)[\rho]\] for all $\rho \in \pplusac{I}$. Whatever $F$ measures of $\rho \in \pplusac{I}$ is given in the same scalar quantity by $\widetilde{F}$ of $\Lbf[\rho] \in \lorenz{I}.$

Riemannian structures will be endowed upon the spaces of probability measures and Lorenz curves. Let $\mathcal{M}$ be a manifold and $T_u\mathcal{M}$ be the tangent space at $u\in\mathcal{M}$ then for $h_1,h_2\in T_u\mathcal{M}$, the Riemannian metric tensor is expressed as $\langle h_1, h_2\rangle _{\mathcal{M},u}.$ 

The burdensome writing of the manifold in question as a subscript to the inner product is necessary in this paper as there will be several variations of similar spaces of functions. This will help us keep track. If there are two broadly similar Riemannian spaces $\mathcal{M}_\alpha$ and $\mathcal{M}_\beta$, which are ``variations on a theme,'' we may denote the metric tensors by $\langle\,,\,\rangle_{\mathcal{M},\alpha,u}$ and $\langle\,,\,\rangle_{\mathcal{M},\beta,u}$ to avoid multiple levels of subscripts.

The canonical $L^2(\Omega)$ inner product will be indicated by $(h_1,h_2)$.

For a functional $F:\mathcal{M}\rightarrow \mathbb{R}$, the gradient of $F$ at $\hat{u}_0\in\mathcal{M}$ is \[\totdd{}{t}\bigg|_{t=0}\left(F\circ u\right)(t) = \left \langle\grad{\mathcal{M}}{F[\hat{u}_o]},\ddt{u_t}\bigg|_{t=0} \right\rangle_{\mathcal{M},\hat{u}_0}\] for all smooth curves $u_t: [-T,T] \rightarrow \mathcal{M}$ passing through $\hat{u}_0$ at $t=0$.

When the manifold is a space of functions over a base space $\Omega \subset \mathbb{R}^n$, we can define the canonical differential, the Fr\'echet derivative of a functional when it exists, to be \[\left( \frechet{F}{u_t}[u_0],\ddt{u_t}\bigg|_{t=0} \right) = \int_\Omega\,dx\,\left(\frechet{F}{u_t}[u_0](x) \right)\left(\ddt{u_t}\bigg|_{t=0}(x)\right).\]

\subsection{Some technical results}\label{ssec:techResults}
Here we state several results that are used across the different sections of the paper. The proofs of the two lemmas are presented in \cref{apdx:techLemmasProofs}.

For a smooth, time-parameterized curve $\rho_t\in \pplusac{\real}$, recall this paper's convention that $C[\rho_t]:\real\rightarrow[0,1]$ and $G[\rho_t]:[0,1]\rightarrow\real$ are the cdf and inverse cdf of $\rho_t$, respectively. Let $F:\pplusac{\real}\rightarrow\real$ and $\widetilde{F}:\lorenz{\real}\rightarrow\real$ be such that $F[\rho] = (\widetilde{F}\circ\Lbf)[\rho]$ for all $\rho\in\pplusac{\real}$.

The infinitesimal perturbation of the density, which integrates to zero over $\real$, is denoted $h(x) = \ddt{}\big|_{t=0}\rho_t(x)$.

The first two results give the variations of the inverse cdf and Lorenz curve under a perturbation of the density.\footnote{If the density, $\rho_t,$ varied according to a McKean-Vlasov equation then \cref{lem:frechetInverseCDF,lem:frechetLorenzMap} would be precursors to the work in \cite{LorDyn2024}.}

\begin{lemma}\label{lem:frechetInverseCDF}
Fix $f\in [0,1]$ then
\begin{equation*}
    \ddt{}\bigg|_{t=0}G[\rho_t](f) = -\frac{\int_{-\infty}^{G[\rho_0](f)}\,dy\,h(y)}{\left(\rho_0\circ G[\rho_0]\right)(f)}.
\end{equation*}    
\end{lemma}

\begin{lemma}\label{lem:frechetLorenzMap}
    Let $f\in[0,1]$ then 
    \begin{equation*}
        \ddt{}\bigg|_{t=0}\Lbf[\rho_t](f) =\int_{-\infty}^{G[\rho_0](f)}\,dy\,yh(y) - G[\rho_0](f)\int_{-\infty}^{G[\rho_0](f)}\,dy\,h(y).
    \end{equation*}
\end{lemma}

The next result can be stated in different sets of variables and both such expressions are instructive to have available. The relationship between the Fr\'echet derivative of $F$ with respect to the probability density and the Fr\'echet derivative of $\widetilde{F}$ with respect to a Lorenz curve is dependent only on the variable transformation between $(x,t,\rho)$ and $(f,t,\LL)$ and is therefore useful across the several metric structures considered.

\begin{proposition}\label{prop:frechetFunctionalChangeOfVar}
For $x\in\real,$
    \begin{equation}
        \frechet{F}{\rho}(x) = - \int_{-\infty}^x\,dy\,\int_{-\infty}^y\,du\,\rho(u)\frechet{\widetilde{F}}{\LL}\left(C[\rho](u)\right)
    \end{equation}
    or, equivalently for $f\in [0,1]$
    \begin{equation}
        \frechet{F}{\rho}\left(G[\rho](f)\right) = -\int_0^f\,dg\,\frac{1}{\rho\left(G[\rho](g)\right)}\int_0^g\,dl\,\frechet{\widetilde{F}}{\LL}(l).
    \end{equation}
\end{proposition}
\begin{proof}
Consider a smooth, time-parameterized curve $\rho_t\in\pplusac{\real}.$ Since \begin{equation*}
        \totdd{}{t}\bigg|_{t=0}F[\rho_t] = \totdd{}{t}\bigg|_{t=0}(\widetilde{F}\circ\Lbf)[\rho_t],
    \end{equation*} we have \begin{equation}\label{eq:frechetRel1COV}
        \int_\real\,dx\,\frechet{F}{\rho_t}[\rho_0](x)h(x) = \int_0^1\,df\,\frechet{\widetilde{F}}{\LL_t}[\LL_0](f)\ddt{}\bigg|_{t=0}\Lbf[\rho_t](f).
    \end{equation}
    We focus on the right hand side of \cref{eq:frechetRel1COV} and invoke \cref{lem:frechetLorenzMap} to obtain \begin{align*}
        \int_0^1\,df\,\frechet{\widetilde{F}}{\LL_t}[\LL_0](f)\ddt{}\bigg|_{t=0}\Lbf[\rho_t](f) &= \int_0^1\,df\,\frechet{\widetilde{F}}{\LL_t}[\LL_0](f)\left[\int_{-\infty}^{G[\rho_0](f)}\,dy\,yh(y) - G[\rho_0](f)\int_{-\infty}^{G[\rho_0](f)}\,dy\,h(y)\right]\\
        &= \int_\real\,dx\,\rho_0(x)\frechet{\widetilde{F}}{\LL_t}[\LL_0]\left(C[\rho_0](x)\right)\left[\int_{-\infty}^x\,dy\,yh(y)-x\int_{-\infty}^x\,dy\,h(y)\right]\\
        &= -\int_\real\,dx\,\left[\int_{-\infty}^x\,du\,\rho_0(u)\frechet{\widetilde{F}}{\LL_t}[\LL_0]\left(C[\rho_0](u)\right)\right]\left[-\int_{-\infty}^x\,dy\,h(y)\right]\\
        &=- \int_\real\,dx\,\left[\int_{-\infty}^x\,dy\,\int_{-\infty}^y\,du\,\rho_0(u)\frechet{\widetilde{F}}{\LL_t}[\LL_0]\left(C[\rho_0](u)\right)\right]h(x),
    \end{align*} assuming the vanishing of the boundary terms from the integration by parts.

    Since the curve $\rho_t\in\pplusac{\real}$ was arbitrary, returning again to \cref{eq:frechetRel1COV} gives the conclusion in the $x\in\real$ variable.

    The integral transformation rules of \cref{sec:reviewOfResultsLorDyn} show the result in the cdf variable of $f.$
\end{proof}

\subsection{A tangent space for the space of Lorenz curves}\label{ssec:LorTanSpace}
The tangent space of $W_2(\Omega)$ at $\rho$, denoted $T_\rho W_2(\Omega)$, may be identified with measurable $h:\Omega\rightarrow\real$ such that $\int_\Omega\,dx\,h(x) = 0.$\footnote{This is natural since (equivalence classes of) the time derivatives of smooth curves are by definition elements of the tangent space and $\dot{\rho}_t$ is mean-zero over $\Omega$.} That is, $h\in T_\rho W_2(\Omega)$ is an admissible infinitesimal perturbation of the probability measure $\rho$ such that both the nonnegativity and the normalization of total probability of the resulting measure remain. A tangent vector $h$ may also be viewed as the difference between the densities of two probability measures.

Below we propose a sensible view of the tangent space over the space of Lorenz curves.

Take $\rho_1, \rho_2 \in \pplusac{\real}$ and set $\LL_i = \Lbf[\rho_i]$ for $i=1,2.$ The perturbation, $\eta:[0,1]\rightarrow\real$, we now consider is defined by $\eta(f) = \LL_2(f) - \LL_1(f)$. The properties held by a generic $\eta$ will characterize the tangent space.

We can quickly note that $\eta(0) = 0$. Next observe that $\totdd{\eta(f)}{f}\big|_{f=0} = 0$ since \begin{align*}
\totdd{}{f}\bigg|_{f=0} \eta(f) &= \totdd{}{f}\bigg|_{f=0} \left(\int_{-\infty}^{G[\rho_2](f)}\,dy\,y\rho_2(y)-\int_{-\infty}^{G[\rho_1](f)}\,dy\,y\rho_1(y) \right)\\
&=\left[ G[\rho_2](f)\rho_2(G[\rho_2](f))\totdd{G[\rho_2](f)}{f}-G[\rho_1](f)\rho_1(G[\rho_1](f))\totdd{G[\rho_1](f)}{f}\right]_{f=0}\\
&=G[\rho_2](0)-G[\rho_1](0)\\
&= 0,
\end{align*}
where we made use of \cref{eq:fDerivL,eq:ffDerivL} in \cref{sec:reviewOfResultsLorDyn} to apply $\totdd{G[\rho](f)}{f} = 1/\rho(G[\rho](f))$ and the positivity of the $\rho_i$.

A similar calculation shows that $\totdd{\eta(f)}{f}\big|_{f=1} = 0$. That said, $f=0,1$ are the only two values for which generic $\eta'(f)$ is zero. Since for arbitrary $f\in (0,1)$, $G[\rho_2](f)$ is not $G[\rho_1](f)$, but by virtue of working with positive measures, the behavior of the inverse cdf at $f=0$ and $1$ is clear.

\begin{definition}\label{def:genericLorenzTangentSpace}
    The set of admissible perturbations of Lorenz curves is the vector space \begin{equation}\label{eq:genericLorenzTangentSpace}
    \LorTan\defeq\left\{\eta:[0,1]\rightarrow\real\,:\, \eta(0)=0,\, \eta'(0) = 0,\, \text{and } \eta'(1) = 0 \right\}.
\end{equation} 
\end{definition} In the coming sections, this set will be endowed with different inner products for the purpose of identifying several novel gradient structures on the space of Lorenz curves.

\section{2-Wasserstein gradient flows}\label{sec:WassersteinGradFlowCompat}

The use of 2-Wasserstein, $W_2$, theory and optimal transport to view evolution equations as the gradient flows of meaningful energy functionals began around the turn of the twenty-first century with the two papers \cite{MR1617171,MR1842429}. For a broader review of the theory of optimal transport and its many applications, see \cite{MR4294651,MR4655923}.

Let $\Omega\subseteq \real^n$ be a measurable, convex set with smooth boundary. $W_2(\Omega)$ is a metric space and a formal infinite-dimensional Riemannian manifold over the set \begin{equation*}
    \ptwoac{\Omega}\defeq \left\{ \rho \in \p{\Omega} \,:\, \text{a.c. w.r.t. Lebesgue and }\int_\Omega\,dx\,\rho(x)|x|^2<\infty\right\}.
\end{equation*}
The metric tensor at $\rho\in\ptwoac{\Omega}$ is \begin{equation}\label{eq:w2metricTensor}
    \left\langle h_1,h_2\right \rangle_{W_2,\rho}\defeq \int_\Omega\,dx\,\rho(x) \nabla\psi_1\cdot\nabla\psi_2 \text{ where for } i=1,2 \, \begin{cases}
\divg(\rho\nabla\psi_i)=-h_i &\text{in } \Omega\\
\ddNorm{\psi_i} = 0 &\text{on } \partial\Omega,
\end{cases}
\end{equation} where the tangent vectors $h_i:\Omega\rightarrow\real$  are measurable and mean zero, in the sense that $\int_\Omega\,dx\,h_i(x)=0.$

The metric tensor, $ \left\langle \, , \,\right \rangle_{W_2,\rho}$, induces a gradient structure such that if $F:\ptwoac{\Omega}\rightarrow\real$ is smooth then \begin{equation*}
    \grad{W_2,\rho}{F[\rho]}= -\divg\left(\rho\nabla\frechet{F}{\rho}\right).
\end{equation*}

Following \cite{MR3150642}, we define the Onsager operator, $\KK_{\text{Otto},\rho}:T^*_\rho W_2(\Omega)\rightarrow T_\rho W_2(\Omega)$ to be \begin{equation}
    \KK_{\text{Otto},\rho}\defeq-\divg(\rho\nabla [\cdot]).
\end{equation} The Onsager operator is a nice way to present the gradient structure of manifolds such as these.

Consider a Fokker-Planck equation \begin{equation}\label{eq:w2gradflowgen}
    \ddt{\rho_t} = \grad{W_2,\rho_t}{F[\rho_t]}
\end{equation} generated by the 2-Wasserstein gradient flow of a functional $F:\ptwoac{I}\rightarrow\mathbb{R}$. We call a metric tensor on the space of Lorenz curves \textit{compatible} with the metric structure of $W_2(I)$ if the gradient flow of $\widetilde{F}:\lorenz{I}\rightarrow\mathbb{R}$ induced by the metric tensor in the space of Lorenz curve generates the same time-evolution as the dynamics given by carrying out the $(x,t,\rho)\mapsto (f,t,\LL)$ variable transformation of \cref{eq:w2gradflowgen}

Said another way, a metric tensor on the space of Lorenz curves will be considered compatible if the connection between the lower left and right equations in \cref{fig:mainPoint1} is completed. In such a case, the same variational principle is established in either set of variables. 

\subsection{Linear mobility}\label{ssec:linMobW2}
\begin{definition}\label{def:tensorLorW2}
    Given $\eta_1,\eta_2 \in \LorTan$ (see \cref{def:genericLorenzTangentSpace}), let \begin{equation}
        \left \langle \eta_1,\eta_2 \right \rangle_{\LLfrak,W_2,\LL} \defeq \int_0^1\,df\, \totdd{\eta_1(f)}{f}\totdd{\eta_2(f)}{f}.
    \end{equation} Let $T_\LL \lorenzsub{W_2}$ denote the tuple of $\LorTan$ and $\left\langle\,,\,\right\rangle_{\LLfrak,W_2,\LL}$.
\end{definition}

\begin{remark}
    Let $I\subseteq \real$. The set $\lorenz{I}$ along with the tangent space from \cref{def:tensorLorW2} will be referred to as $\lorenzsub{W_2}(I)$ or sometimes just $\lorenzsub{W_2}.$
\end{remark}

$T_\LL \lorenzsub{W_2}$ is an inner product space, which we refer to as the tangent space of $\lorenzsub{W_2}$ at $\LL\in \lorenz{I}$. Frequently we will abuse notation and refer to the elements of $\LorTan$ as the elements of $T_\LL \lorenzsub{W_2}$.\footnote{The inner product does not yet depend on the basepoint $\LL$ but it is worthwhile to have this notation set from the beginning.}

\begin{lemma}[First-order functional calculus]\label{lem:firstOrderCalcW2Lor}
    Let $\widetilde{F}:\lorenz{\real}\rightarrow\real$ then \begin{equation}
        \grad{{\LLfrak,W_2,\LL}}{\widetilde{F}} = \psi,
    \end{equation} where \begin{equation}\tottwodd{\psi}{f} = - \frechet{\widetilde{F}}{\LL}.\end{equation}
\end{lemma}

One interpretation is that the Onsager operator of $\lorenzsub{W_2}$ acts as the inverse of the negative Poisson operator with datum $\frechet{\widetilde{F}}{\LL}.$ Moreover, since the gradient is a tangent space element, the boundary conditions on $\psi:[0,1]\rightarrow\real$ are that $\psi(0), \psi'(0),$ and $\psi'(1)$ are all zero.

\begin{proof}
Let $T>0$ and $\LL_t: [-T,T] \rightarrow \lorenz{\real}$ be a smooth curve. The rate of change of $\widetilde{F}\circ \LL_t$ at $t=0$ can be expressed in two equivalent ways: \begin{equation}\label{eq:rocW2Linear}\left(\frechet{\widetilde{F}}{\LL_t}[\LL_0], \ddt{\LL_t}\bigg|_{t=0}\right) = \dd{\left(\widetilde{F}\circ \LL_t\right)}{t}\bigg|_{t=0} =\left\langle\grad{\LLfrak,W_2,\LL_0}{\widetilde{F}}, \ddt{\LL_t}\bigg|_{t=0}\right\rangle_{\LLfrak,W_2,\LL_0}.\end{equation}

First, expand the r.h.s. of \cref{eq:rocW2Linear} using \cref{def:tensorLorW2}, \begin{align}
 \dd{\left(\widetilde{F}\circ \LL_t\right)}{t}\bigg|_{t=0} &= \left(\ddf{}\grad{\LLfrak,W_2,\LL_0}{\widetilde{F}}, \ddf{}\ddt{\LL_t}\bigg|_{t=0}\right)\\
&= -\left(\twoddf{}\grad{\LLfrak,W_2,\LL_0}{\widetilde{F}}, \ddt{\LL_t}\bigg|_{t=0}\right).\label{eq:rocW2Linear2}
\end{align} Comparing \cref{eq:rocW2Linear2} to the lh.s. of \cref{eq:rocW2Linear} implies that \[\frechet{\widetilde{F}}{\LL_t}[\LL_0] = - \twoddf{}\grad{\LLfrak,W_2,\LL_0}{\widetilde{F}}.\]

Let $\psi = \grad{\LLfrak,W_2,\LL_0}{\widetilde{F}}$ then $\psi$ satisfies \[\twoddf{\psi} = - \frechet{\widetilde{F}}{\LL_t}[\LL_0],\] as we sought.
\end{proof}

\begin{theorem}\label{thm:varEquivw2Lor}
    For $F:\ptwoac{\real}\rightarrow\real$, the $(f,t,\LL)$-transformed dynamics of \begin{equation*}
        \ddt{\rho_t} = \grad{W_2,\rho_t}{F[\rho_t]}
    \end{equation*} are equivalent to those of \begin{equation*}
        \ddt{\LL_t} = \grad{\LLfrak,W_2,\LL_t}{\widetilde{F}[\LL_t]},
    \end{equation*}
    where $\widetilde{F}$ satisfies $\left(\widetilde{F}\circ\Lbf\right)[\rho]=F[\rho]$ for any $\rho\in\ptwoac{\real}$.
\end{theorem}

\begin{proof}
    Beginning with the $W_2$ gradient dynamics \begin{equation*}
        \ddt{\rho_t(x)}=-\ddx{}\left[\rho_t(x)\ddx{}\frechet{F}{\rho_t}(x)\right],
    \end{equation*} multiply both sides by $x$ and integrate to obtain \begin{equation*}
        \ddt{L_t(x)} = -\int_{-\infty}^x\,dy\,y\ddy{}\left[\rho_t(y)\ddy{}\frechet{F}{\rho_t}(y)\right],
    \end{equation*} where $L_t(x)= \int_{-\infty}^x\,dy\,y\rho_t(y).$

    Recall that $C[\rho_t]$ and $G[\rho_t]$ denote the cdf and inverse cdf, respectively, of $\rho_t.$

    Switching to the cdf integration variables, the last expression becomes\footnote{This step simultaneously invokes the integral and differential transformations of \cref{sec:reviewOfResultsLorDyn}. The arguments to the function $\rho_t$ are $G[\rho_t](g)$, which are suppressed for clarity of the intermediate computations.}
    \begin{equation*}
        \ddt{L_t(x)} = -\int_{0}^{C[\rho_t](x)}\,dg\,G[\rho_t](g)\frac{1}{\rho_t}\rho_t\ddg{}\left[(\rho_t)^2\ddg{}\frechet{F}{\rho_t}\left(G[\rho_t](g)\right)\right],
    \end{equation*}
    Next we insert the relationship between $\frechet{F}{\rho}$ and $\frechet{\widetilde{F}}{\LL}$ derived in \cref{prop:frechetFunctionalChangeOfVar} and simplify so that \begin{equation*}
        \ddt{L_t(x)} = \int_{0}^{C[\rho_t](x)}\,dg\,G[\rho_t](g)\ddg{}\left[\rho_t\int_0^g\,dl\,\frechet{\widetilde{F}}{\LL_t}(l)\right].
    \end{equation*} After which integration by parts and assuming the vanishing of boundary terms gives \begin{equation*}
        \ddt{L_t(x)} = -\int_{0}^{C[\rho_t](x)}\,dg\,\ddg{G[\rho_t](g)}\rho_t\int_0^g\,dl\,\frechet{\widetilde{F}}{\LL_t}(l).
    \end{equation*}

    Observing that $\ddg{G[\rho](g)} = 1/\rho$ and replacing $x$ with $G[\rho_t](f)$ gives \begin{equation}\label{eq:Ltdynproof1}
        \ddt{\LL_t(f)} = -\int_0^f\,dg\,\int_0^g\,dl\,\frechet{\widetilde{F}}{\LL_t}(l),
    \end{equation} since $\LL_t(f):=\left(L_t\circ G[\rho_t]\right)(f).$

    The proof is concluded by noting that the right hand side of \cref{eq:Ltdynproof1} is precisely the first-order calculus, $\grad{\LLfrak,W_2,\LL_t}{\widetilde{F}[\LL_t]}$, of $\lorenzsub{W_2}$ derived in \cref{lem:firstOrderCalcW2Lor}.
\end{proof}

\begin{example}
    Let $V:\real\rightarrow\real$ be a $C^1$ scalar potential function and $W:\real \times \real \rightarrow\real$  be a symmetric and $C^1$ function, which should be viewed as the interaction energy density of particles at positions $x$ and $y$. The $W_2$ gradient flow of the classical free-energy function, $F=\mathcal{V}+\mathcal{W}-S,$ given by a potential energy, an interaction energy, and Boltzmann entropy \begin{equation*}
        F[\rho] = \int_\real\,dx\,\rho(x)V(x)+\frac{1}{2}\int_\real\,dx\,\int_\real\,dy\,W(x,y)\rho(x)\rho(y)+\int_\real\,dx\,\rho(x)\log \rho(x)
    \end{equation*}
    has Lorenz dynamics in the $(f,t,\LL)$ variables given by the $\lorenzsub{W_2}$ gradient flow of 
    \begin{equation*}
       \widetilde{F}[\LL] = \int_0^1\,df\, V\left(\totdd{\LL}{f}\right)+\frac{1}{2}\int_0^1\,df\,\int_0^1\,dg\, W\left(\totdd{\LL}{f},\totdd{\LL}{g}\right)-\int_0^1\,df\, \log \left(\tottwodd{\LL}{f}\right).
    \end{equation*}
\end{example}

\subsection{Non-linear mobility}\label{ssec:nonlinMobW2M}
In the 2010s, an adaptation of the classic 2-Wasserstein space emerged that modified the dynamic formulation of the metric to include a non-linear mobility term \cite{MR2565840,MR2448650,MR2672546,MR2921215}. By including such a term, the modified geometry of the space was able to better reflect kinetic premises in which the ease of particles transiting through a domain may depend on the local density of particles in a non-linear way.

In this section, we modify the definition of the metric tensor defined on the tangent space of $\lorenz{I}$, and show that the variational principles of $W_{2,M}$ are compatible under the modified geometry imposed on the space of Lorenz curves.

Let $M(\rho) = m(\rho)\rho$, where $m:\posreal\rightarrow\posreal.$
\begin{definition}[Metric tensor]\label{def:tensorLorW2M}
    Given $\eta_1,\eta_2 \in \LorTan$ , define \begin{equation}
        \left \langle \eta_1,\eta_2 \right \rangle_{\LLfrak,W_{2,M},\LL} \defeq \int_0^1\,df\, \left(m\left(\frac{1}{\tottwodd{\LL}{f}}\right)\right)^{-1}\totdd{\eta_1(f)}{f}\totdd{\eta_2(f)}{f}.
    \end{equation}
    Let $T_\LL \lorenzsub{W_{2,M}}$ denote the pair $\LorTan$ and $\left \langle \, , \, \right \rangle_{\LLfrak,W_{2,M},\LL}$
\end{definition}

\begin{remark}
    Let $I\subseteq \real$. The set $\lorenz{I}$ along with the tangent space from \cref{def:tensorLorW2M} will be referred to as $\lorenzsub{W_{2,M}}(I)$ or sometimes just $\lorenzsub{W_{2,M}}.$
\end{remark}

\begin{lemma}[First-order functional calculus]\label{lem:firstOrderCalcW2MLor}
    Let $\widetilde{F}:\lorenz{\real}\rightarrow\real$ then \begin{equation}
        \grad{{\LLfrak,W_{2,M},\LL}}{\widetilde{F}} = \psi,
    \end{equation} where \begin{equation*}\totdd{}{f}\left(\frac{1}{m}\totdd{}{f}\psi\right) = - \frechet{\widetilde{F}}{\LL}\end{equation*} and the argument of $m$ is $\left(\tottwodd{\LL(f)}{f}\right)^{-1}$.
\end{lemma}

\begin{proof}
Let $T>0$ and $\LL_t: [-T,T] \rightarrow \lorenz{\real}$ be a smooth curve. Then \begin{equation}\label{eq:rocW2M}\left(\frechet{\widetilde{F}}{\LL_t}[\LL_0], \ddt{\LL_t}\bigg|_{t=0}\right) = \dd{\left(\widetilde{F}\circ \LL_t\right)}{t}\bigg|_{t=0} =\left\langle\grad{\LLfrak,W_{2,M},\LL_0}{\widetilde{F}}, \ddt{\LL_t}\bigg|_{t=0}\right\rangle_{\LLfrak,W_{2,M},\LL_0}.\end{equation}

Focusing first on the $L^2$ expression, in \cref{eq:rocW2M}, for the rate of change, note that \begin{align}
\left(\frechet{\widetilde{F}}{\LL_t}[\LL_0], \ddt{\LL_t}\bigg|_{t=0}\right) &= -\int_0^1\,df\,\left(\int_0^f\,dg\,\frechet{\widetilde{F}}{\LL_t}[\LL_0]\right)\left(\ddf{}\ddt{\LL_t}\bigg|_{t=0}\right) \nonumber\\
&=\int_0^1\,df\,\left[\int_0^f\,dg\,m\left(\frac{1}{\twoddg{\LL_0}}\right)\left(\int_0^g\,dh\,\frechet{\widetilde{F}}{\LL_t}[\LL_0]\right)\right]\nonumber \\
&\qquad\times\ddf{}\left[\left(m\left(\frac{1}{\twoddf{\LL_0}}\right)\right)^{-1}\left(\ddf{}\ddt{\LL_t}\bigg|_{t=0}\right)\right].\label{eq:rocW2ML2v2}
\end{align}

Turning to the r.h.s. of \cref{eq:rocW2M} and expanding it using \cref{def:tensorLorW2M} gives
\begin{align}
    \left\langle\grad{\LLfrak,W_{2,M},\LL_0}{\widetilde{F}}, \ddt{\LL_t}\bigg|_{t=0}\right\rangle_{\LLfrak,W_{2,M},\LL_0} &= \int_0^1\,df\,\left(m\left(\frac{1}{\twoddf{\LL_0}}\right)\right)^{-1}\left(\ddf{}\grad{\LLfrak,W_{2,M},\LL_0}{\widetilde{F}}\right)\left(\ddf{}\ddt{\LL_t}\bigg|_{t=0}\right)\nonumber\\
    &= -\int_0^1\,df\,\grad{\LLfrak,W_{2,M},\LL_0}{\widetilde{F}}\ddf{}\left[\left(m\left(\frac{1}{\twoddf{\LL_0}}\right)\right)^{-1}\left(\ddf{}\ddt{\LL_t}\bigg|_{t=0}\right)\right]\label{eq:rocW2MLorExpr}.
\end{align}

Combining \cref{eq:rocW2ML2v2,eq:rocW2MLorExpr}, can conclude that \begin{equation}
    \grad{\LLfrak,W_{2,M},\LL_0}{\widetilde{F}} = - \left[\int_0^f\,dg\,m\left(\frac{1}{\twoddg{\LL_0}}\right)\left(\int_0^g\,dh\,\frechet{\widetilde{F}}{\LL_t}[\LL_0]\right)\right].
\end{equation}

With an eye towards the linear mobility case, observe that $\psi= \grad{\LLfrak,W_{2,M},\LL_0}{\widetilde{F}}$ satisfies \begin{equation*}
    \ddf{}\left(\frac{1}{m}\ddf{}\psi\right) = -\frechet{\widetilde{F}}{\LL_t}.
\end{equation*}

\end{proof}

\begin{theorem}\label{thm:varEquivw2MLor}
    For $F:\ptwoac{\real}\rightarrow\real$, the Lorenz dynamics in the variables $(f,t,\LL)$ of the $W_{2,M}$ gradient flow of $F$ \begin{equation*}
        \ddt{\rho_t} = \grad{W_{2,M},\rho_t}{F[\rho_t]}
    \end{equation*} are equivalent to those of \begin{equation*}
        \ddt{\LL_t} = \grad{\LLfrak,W_{2,M},\LL_t}{\widetilde{F}[\LL_t]},
    \end{equation*}
    where $\widetilde{F}$ satisfies $\left(\widetilde{F}\circ\Lbf\right)[\rho]=F[\rho]$ for any $\rho\in\ptwoac{\real}$.
\end{theorem}

\begin{proof}
    The proof follows a similar progression as that of \cref{thm:varEquivw2Lor}.
\end{proof}

\section{Isometries between the Wasserstein spaces and the associated spaces of Lorenz curves}\label{sec:w2MIsom}
The results of \cref{sec:WassersteinGradFlowCompat} show that the variational principles of 2-Wasserstein spaces are maintained in the space of Lorenz curves so long as the appropriate metric tensor is selected. The interesting properties of the metric tensors defined above do not end there. In fact, the transformation between the spaces is an isometry in the sense of lengths induced by the metric tensors.

For $\rho_0, \rho_1 \in \ptwoac{\real}$, let $\Gamma(\rho_0, \rho_1)$ be the couplings of $\rho_0$ and $\rho_1$, that is joint measures on $\real \times \real$ with marginals $\rho_0$ and $\rho_1$, respectively.

The 2-Wasserstein optimal transport metric
\begin{equation*}
    d_{W_2}^2(\hat{\rho}_0,\hat{\rho}_1) = \inf_{\gamma \in \Gamma(\hat{\rho}_0,\hat{\rho}_1)}\int_{\real \times \real}\,d\gamma(x,y)\, |x-y|^2
\end{equation*} can be expressed using the celebrated Benamou-Brenier formula \cite{MR4655923} as \begin{equation*}
        d_{W_{2}}^2(\hat{\rho}_0,\hat{\rho}_1) = \inf_{\text{curves } \rho_t} \left\{\int_0^1\,dt\,\int_\real\,dx\,\rho_t \left(\ddx{\psi_t} \right)^2 \text{where } \ddt{\rho_t} = -\ddx{}\left(\rho_t\ddx{\psi_t}\right)\,:\, \rho_0=\hat{\rho}_0,\, \rho_1=\hat{\rho}_1\right\}.
\end{equation*} Since the Benamou-Brenier formula is an expression of action minimization, a formal tangent space norm may be deduced. In this case, the norm is that of a measure-weighted, negative first-order, homogeneous Sobolev space, which may be seen in \cref{eq:w2metricTensor}, the definition of $ \left\langle \, , \,\right \rangle_{W_2,\rho}$.

\begin{definition}[$\lorenzsub{W_2}$ global metric]
Let $||\cdot||_{\LLfrak,W_2,\LL}$ be the norm induced by \cref{def:tensorLorW2}. Then between two Lorenz curves $\hat{L}_0$ and $\hat{L}_1$, 
    \begin{align*}\label{eq:globalLorW2metric}
        d^2_{\LLfrak,W_2}(\hat{L}_0,\hat{L}_1) &= \inf_{\text{curves } \LL_t} \left\{\int_0^1\,dt\,||\dot{\LL}_t||^2_{\LLfrak,W_2,\LL_t} \, :\, \LL_0 = \hat{L}_0, \, \LL_1 = \hat{L}_1\right\}\\
        &=\inf_{\text{curves } \LL_t} \left\{\int_0^1\,dt\,\int_0^1\,df\,\left(\ddf{\dot{\LL}_t(f)}\right)^2 \, :\, \LL_0 = \hat{L}_0, \, \LL_1 = \hat{L}_1\right\}.
    \end{align*}
\end{definition}

\begin{proposition}[$W_{2}$ and $\lorenzsub{W_{2}}$ isometry]\label{prop:w2LorIsom}
    For two positive distributions $\hat{\rho}_0$ and $\hat{\rho}_1$ over $\real$, \[d_{W_2}^2(\hat{\rho}_0,\hat{\rho}_1) = d_{\LLfrak,{W_2}}^2(\Lbf[\hat{\rho}_0],\Lbf[\hat{\rho}_1]).\]
\end{proposition}

We can state and prove \cref{prop:w2LorIsom} for the more general case of possibly nonlinear mobility. Recall that we define $m(\rho)>0$ via $M(\rho) = m(\rho)\rho.$

\begin{definition}[$\lorenzsub{W_{2,M}}$ global metric]
Let $||\cdot||_{\LLfrak,W_{2,M},\LL}$ be the norm induced by \cref{def:tensorLorW2M}. Then between two Lorenz curves $\hat{L}_0$ and $\hat{L}_1$, 
    \begin{align}
        d^2_{\LLfrak,W_{2,M}}(\hat{L}_0,\hat{L}_1) &= \inf_{\text{curves } \LL_t} \left\{\int_0^1\,dt\,||\dot{\LL}_t||^2_{\LLfrak,W_{2,M},\LL_t} \, :\, \LL_0 = \hat{L}_0, \, \LL_1 = \hat{L}_1\right\}\nonumber\\
        &=\inf_{\text{curves } \LL_t} \left\{\int_0^1\,dt\,\int_0^1\,df\,\left(m\left(\frac{1}{\twoddf{\LL_t(f)}}\right)\right)^{-1}\left(\ddf{\dot{\LL}_t(f)}\right)^2 \, :\, \LL_0 = \hat{L}_0, \, \LL_1 = \hat{L}_1\right\}.\label{eq:LorW2MGlobalMetricv2}
    \end{align}
\end{definition}

\begin{remark}[The Generalized Benamou-Brenier Formula]
    The non-linear mobility 2-Wasserstein distance may be expressed as \begin{equation}
        d_{W_{2,M}}^2(\hat{\rho}_0,\hat{\rho}_1) = \inf_{\text{curves } \rho_t} \left\{\int_0^1\,dt\,\int_\real\,dx\,M(\rho_t) \left(\ddx{\psi_t} \right)^2 \text{where } \ddt{\rho_t} = -\ddx{}\left(M(\rho_t)\ddx{\psi_t}\right)\,:\, \rho_0=\hat{\rho}_0,\, \rho_1=\hat{\rho}_1\right\}.\label{eq:w2mBenamouBrenier}
    \end{equation}
\end{remark}

\begin{theorem}[$W_{2,M}$ and $\lorenzsub{W_{2,M}}$ isometry]\label{thm:w2MLorIsom}
    For two positive distributions $\hat{\rho}_0$ and $\hat{\rho}_1$ over $\real$, \[d_{W_{2,M}}^2(\hat{\rho}_0,\hat{\rho}_1) = d_{\LLfrak,{W_{2,M}}}^2(\Lbf[\hat{\rho}_0],\Lbf[\hat{\rho}_1]).\]
\end{theorem}

\begin{proof}
    Consider a time-parameterized curve $\LL_t$ in $\lorenz{\real}$ for $t\in[0,1]$. For the associated positive densities, $\rho_t = \Pbf[\LL_t]$, let $\psi_t$ satisfy \begin{equation}
        \ddt{\rho_t} = -\ddx{}\left(M(\rho_t)\ddx{\psi_t}\right)\label{eq:psiTEllipticPDEw2misom}
    \end{equation} and vanish as $x\rightarrow\pm\infty$.

    We focus on the term \begin{equation*}
        \ddf{\dot{\LL}_t(f)}
    \end{equation*} in the integrand of the $\lorenzsub{W_{2,M}}$ global metric, \cref{eq:LorW2MGlobalMetricv2}, with the intent to express it using $\psi_t.$ Assume sufficient smoothness to exchange the $t$ and $f$ derivatives so \begin{align*}
        \ddf{\dot{\LL}_t(f)}&= \ddt{}\ddf{}\LL_t\\
        &= \ddt{G[\rho_t](f)},
    \end{align*} by \cref{eq:fDerivL} in \cref{sec:reviewOfResultsLorDyn}.
    The time derivative of $G[\rho_t](f)$, also in \cref{sec:reviewOfResultsLorDyn}, is \begin{equation}
        \ddt{G[\rho_t](f)}= -\frac{1}{\rho_t(G[\rho_t](f))}\int_{-\infty}^{G[\rho_t](f)}\,dy\,\ddt{\rho_t}\nonumber,
    \end{equation} and upon inserting the expression that $\psi_t$ satisfies, given by  \cref{eq:psiTEllipticPDEw2misom}, we obtain\begin{align}
        \ddt{G[\rho_t](f)} &= \frac{1}{\rho_t(G[\rho_t](f))}M\left(\rho_t\left(G[\rho_t](f)\right)\right)\left[\ddx{}\psi_t\right]_{x=G[\rho_t](f)}\nonumber\\
        &= m(\rho_t(G[\rho_t](f)))\left[\ddx{}\psi_t\right]_{x=G[\rho_t](f)}.
    \end{align}

    Thus, for $\LL_t:[0,1]\rightarrow \lorenz{\real}$, we can re-write the infimization objective of \cref{eq:LorW2MGlobalMetricv2}, \begin{equation}
        \int_0^1\,dt\,\int_0^1\,df\,\left(m\left(\frac{1}{\twoddf{\LL_t(f)}}\right)\right)^{-1}\left(\ddf{\dot{\LL}_t(f)}\right)^2\nonumber,
    \end{equation} using the expression just derived, \begin{equation}
        \ddf{\dot{\LL}_t(f)} = m(\rho_t(G[\rho_t](f)))\left[\ddx{}\psi_t\right]_{x=G[\rho_t](f)}\nonumber.
    \end{equation}
    From \cref{sec:reviewOfResultsLorDyn}, $\rho_t(G[\rho_t](f)) = \frac{1}{\twoddf{\LL_t}}$ so we have \begin{align}
        &\int_0^1\,dt\,\int_0^1\,df\,\left(m\left(\frac{1}{\twoddf{\LL_t(f)}}\right)\right)^{-1}\left(\ddf{\dot{\LL}_t(f)}\right)^2 \nonumber\\ 
        &= \int_0^1\,dt\,\int_0^1\,df\,m\left(\rho_t(G[\rho_t](f))\right)\left(\left[\ddx{}\psi_t\right]_{x=G[\rho_t](f)}\right)^2 \label{eq:w2mIsomProofStep1}\\
        &= \int_0^1\,dt\,\int_\real\,dx\,\rho_t(x)m\left(\rho_t(x)\right)\left(\ddx{\psi_t}\right)^2\label{eq:w2mIsomProofStep2}\\
        &= \int_0^1\,dt\,\int_\real\,dx\,M\left(\rho_t(x)\right)\left(\ddx{\psi_t}\right)^2\label{eq:w2mIsomProofStep3},
    \end{align} where the step from \cref{eq:w2mIsomProofStep1} to \cref{eq:w2mIsomProofStep2} comes from transforming the integral over $f\in[0,1]$ to an integral over $x\in\real$, using the integral transformation rules.

    Infimizing \cref{eq:w2mIsomProofStep3} over all such smooth curves $\rho_t$ with the prescribed terminal points shows the equality with \cref{eq:w2mBenamouBrenier}.
    
    \end{proof}

\section{Adapted Wasserstein gradient flows}\label{sec:cdSpaceTheory}

The $\cdspace$ spaces studied in \cite{CDSpaces2024} are adaptations of the $W_{2,M}$ metric geometry imposed on the space of probability measures. These spaces found success in expressing some McKean-Vlasov evolution equations as gradient flows of meaningful functionals.

In particular, the definition of the $\cdspace$ space metric structures was motivated by a desire to encapsulate \textit{a priori} known conserved quantities, in addition to the conservation of total probability mass. In so doing, a collection of metric tensors was found for the space of positive probability measures over $\real$ with fixed first moment.

Let $D[x,\rho]: \real \times \pplusac{\real}\rightarrow\posreal$.

\begin{definition}\label{def:w22Dtensor}
    At each $\rho\in\pplusac{\real}$, the metric tensor $\langle\,,\,\rangle_{\rho,\cdspace}$ is defined for $h_1,h_2 \in T_\rho \cdspace$ as \begin{equation}\label{eq:w22Dtensor}
\langle h_1,h_2\rangle_{\rho,\cdspace} = \int_\Omega\,dx\,D[x,\rho]\twoddx{\psi_1}\twoddx{\psi_2}
\text{ where for } i =1,2 \, \begin{cases}
\twoddx{}\left(D[x,\rho]\twoddx{\psi_i}\right)=h_i \\
\psi_i(x) \rightarrow 0 \text{ as } x\rightarrow\pm\infty.
\end{cases}\nonumber\end{equation}
\end{definition}

The admissible $h:\real\rightarrow\real$ in the definition of $\langle \, , \,\rangle_{\rho,\cdspace}$ satisfy \[\int_\real\,dx\,h(x) = 0 = \int_\real\,dx\,xh(x).\]

\begin{proposition}[The first-order calculus for $\cdspace(\real)$, \cite{CDSpaces2024}]\label{prop:cdGradFlow}
    For $H:\ptwoac{\real} \rightarrow\real$, \begin{equation}\label{eq:w22dgrad1d}\grad{\cdspace,\rho}{H[\rho]} = \twoddx{}\left(D[x,\rho]\twoddx{}\frechet{H}{\rho}\right).\nonumber\end{equation}
\end{proposition} 

A class of nonlinear partial integro-differential equations arising from the application of mean-field and kinetic theory to idealized economic models are expressible as the gradient flows of a common measure of economic inequality using this family of metric tensors.

\subsection{The case of \texorpdfstring{$\cdspacearg{\rho}$}{C p}}
We start with the case $D[x,\rho] = \rho$, which is analogous to that of linear mobility, $M(\rho) = \rho$, for the Wasserstein spaces.

Let $\LorTan_0 = \{\eta\in\LorTan\,:\,\eta(1) = 0\}$, where $\LorTan$ is defined in \cref{def:genericLorenzTangentSpace}.

\begin{definition}[Metric tensor]
    For $\eta_1,\eta_2 \in \LorTan_0$, let \begin{equation}
        \left \langle \eta_1,\eta_2 \right \rangle_{\LLfrak,\cdspacearg{\rho},\LL} \defeq \int_0^1\,df\,\left(\tottwodd{\LL(f)}{f}\right)^2\eta_1(f)\eta_2(f).
    \end{equation} The tuple $\{\LorTan_0, \langle\,,\,\rangle_{\LLfrak,\cdspacearg{\rho},\LL}\}$ is the tangent space $T_\LL \lorenzsub{\cdspacearg{\rho}}$.
\end{definition}

\begin{lemma}[First-order functional calculus]
    Let $\widetilde{F}:\lorenz{\real}\rightarrow\real$ then \begin{equation}
        \grad{{\LLfrak,\cdspacearg{\rho},\LL}}{\widetilde{F}} = \left(\tottwodd{\LL}{f}\right)^{-2}\frechet{\widetilde{F}}{\LL}
    \end{equation}
\end{lemma}
\begin{proof}
    This is a straightforward application of the techniques used above.
\end{proof}

\begin{proposition}\label{prop:varEquivCrhoLor}
    Consider $F:\pplusac{\real}\rightarrow\real$ and $\widetilde{F}:\lorenz{\real}\rightarrow\real$ such that $\left(\widetilde{F}\circ \Lbf \right)[\rho] = F[\rho]$ for $\rho\in\pplusac{\real}.$ The $(f,t,\LL)$-transformed dynamics of \begin{equation*}
        \ddt{\rho_t} = \grad{\cdspacearg{\rho},\rho_t}{F[\rho_t]}
    \end{equation*} are equivalent to \begin{equation*}
        \ddt{\LL_t}=\grad{\LLfrak,\cdspacearg{\rho},\LL_t}{\widetilde{F}[\LL_t]}.
    \end{equation*}
\end{proposition}

We delay the proof until the next section in which we prove a more general case, encompassing \cref{prop:varEquivCrhoLor}.

\subsection{Generalizing to \texorpdfstring{$\cdspace$}{C D}}
Let $D[x,\rho] = d[x,\rho]\rho$, where $d[x,\rho]>0$. This notation is meant to mirror the nonlinear mobility case of Wasserstein space. Our convention in what follows is that, using the relations of \cref{sec:reviewOfResultsLorDyn}, \begin{equation*}
    \widetilde{d}[f,\LL] = d\left[\totdd{\LL}{f},1\bigg/\tottwodd{\LL}{f}\right],
\end{equation*} which equals $d[x,\rho]$ when evaluated at $f=C[\rho](x).$ This convention extends to $D[x,\rho]$ as well. Therefore we also have that $\widetilde{D}[f,\LL]=\widetilde{d}[f,\LL]\bigg/\tottwodd{\LL(f)}{f}$.

\begin{definition}[Metric tensor]\label{def:tensorLorCD}
    Given $\eta_1,\eta_2 \in \LorTan_0$, let \begin{equation}
        \left \langle \eta_1,\eta_2 \right \rangle_{\LLfrak,\cdspace,\LL} \defeq \int_0^1\,df\,\frac{1}{\widetilde{d}[f,\LL]}\left(\tottwodd{\LL(f)}{f}\right)^2\eta_1(f)\eta_2(f).
    \end{equation}
\end{definition}

\begin{lemma}[First-order functional calculus]\label{lem:firstOrderCalcCDLor}
    Let $\widetilde{F}:\lorenz{\real}\rightarrow\real$ then \begin{align}
        \grad{{\LLfrak,\cdspace,\LL}}{\widetilde{F}} &= \widetilde{d}[f,\LL]\left(\tottwodd{\LL}{f}\right)^{-2}\frechet{\widetilde{F}}{\LL}\\
        &=\widetilde{D}[f,\LL]\left(\tottwodd{\LL}{f}\right)^{-1}\frechet{\widetilde{F}}{\LL}
    \end{align}
\end{lemma}

\begin{theorem}
    \label{thm:varEquivCDLor}
    Consider $F:\pplusac{\real}\rightarrow\real$ and $\widetilde{F}:\lorenz{\real}\rightarrow\real$ such that $\left(\widetilde{F}\circ \Lbf \right)[\rho] = F[\rho]$ for $\rho\in\pplusac{\real}.$ The $(f,t,\LL)$-transformed dynamics of \begin{equation*}
        \ddt{\rho_t} = \grad{\cdspace,\rho_t}{F[\rho_t]}
    \end{equation*} are equivalent to \begin{equation*}
        \ddt{\LL_t}=\grad{\LLfrak,\cdspace,\LL_t}{\widetilde{F}[\LL_t]}.
    \end{equation*}
\end{theorem}

\begin{proof}
    Starting with a $\cdspace$ gradient flow of $F$, given by \cref{prop:cdGradFlow}, \begin{equation*}
        \ddt{\rho_t} = \twoddx{}\left[D[x,\rho_t]\twoddx{}\frechet{F}{\rho}(x)\right],
    \end{equation*} multiply by the spatial variable and integrate to obtain \begin{equation}\label{eq:cdGradFlowLtx}
        \ddt{L_t(x)} = \int_{-\infty}^x\,dy\,y\twoddy{}\left[D[y,\rho_t]\twoddy{}\frechet{F}{\rho}(y)\right],
    \end{equation} where $L_t(x) = \int_{-\infty}^x\,dy\,y\rho_t(y).$

    Transforming the right hand side of \cref{eq:cdGradFlowLtx} into an integration over the cdf variable $g$ yields \begin{equation}\label{eq:cdGradFlowLtxRHSv1}
        \int_0^{C[\rho_t](x)}\,dg\,G[\rho_t](g)\ddg{}\left[\rho\ddg{}D\rho\ddg{}\rho\ddg{}\left(\frechet{F}{\rho}\circ G[\rho_t]\right)(g)\right],
        \end{equation} where the arguments of $\rho$ and $D$ are in the transformed variables but are omitted for clarity.

        Inserting the relationship between $\frechet{F}{\rho}$ and $\frechet{\widetilde{F}}{\LL}$ as derived in \cref{prop:frechetFunctionalChangeOfVar} and simplifying transforms \cref{eq:cdGradFlowLtxRHSv1} into the simpler\begin{equation} \label{eq:cdGradFlowLtxRHSv2}
            -\int_0^{C[\rho_t](x)}\,dg\,G[\rho_t](g)\ddg{}\left[\rho\ddg{}D\rho\frechet{\widetilde{F}}{\LL}\right].
        \end{equation}
        Assuming that the boundary terms vanish from integration by parts, \cref{eq:cdGradFlowLtxRHSv2} becomes just \begin{equation*}
            D\left[x,\rho_t\right]\rho(x)\left(\frechet{\widetilde{F}}{\LL}\circ C[\rho_t]\right)(x),
        \end{equation*} since $\ddg{}G[\rho_t](g)$ and $\rho_t$ cancel and the fundamental theorem of calculus may then be invoked.

        Therefore we have shown that \cref{eq:cdGradFlowLtx} is equivalent to \begin{equation} \label{eq:cdGradFlowLtxv2}
            \ddt{L_t(x)} = D\left[x,\rho_t\right]\rho(x)\left(\frechet{\widetilde{F}}{\LL}\circ C[\rho_t]\right)(x).
        \end{equation}

        Finally, replacing $x$ in \cref{eq:cdGradFlowLtxv2} with $G[\rho_t](f)$ gives \begin{equation*}
            \ddt{\LL_t(f)} = \widetilde{D}[f,\LL_t]\left(\twoddf{\LL_t}\right)^{-1}\frechet{\widetilde{F}}{\LL}(f),
        \end{equation*} since $\LL_t(f) := (L_t\circ G[\rho_t])(f)$, $\widetilde{D}$ is the transformed copy of $D$ in the $(f,t,\LL)$ variables, and $\rho(G[\rho_t](f)) = \left(\twoddf{\LL_t}\right)^{-1}$. This is precisely the $(\LLfrak, \cdspace)$ gradient flow of $\widetilde{F}:\lorenz{\posreal}\rightarrow\real$, \begin{equation*}
            \grad{\LLfrak,\cdspace,\LL_t}{\widetilde{F}[\LL_t]},
        \end{equation*} identified in \cref{lem:firstOrderCalcCDLor}.
\end{proof}

\subsection{The connection to econophysics}

An example of applying \cref{thm:varEquivCDLor} may be found in conjunction with the work in \cite{CDSpaces2024}. Let $D[x,\rho]$ be the time-infinitesimal expected collisional variance for an idealized economic agent of wealth $x$ interacting with the distribution, $\rho_t$, under the diffusion approximation to an unbiased, wealth conserving, positivity preserving kinetic asset exchange model. Then the distribution of wealth of such a mean-field model evolves as a curve of steepest ascent of the (scaled) Gini coefficient, $\mathscr{G},$ of economic inequality\footnote{To be precise, a linearly scaled version of the classical Gini coefficient restricted to distributions over $\posreal$ with unit first moment.}, \begin{equation*}
    \ddt{\rho_t} = +\grad{\cdspace,\rho_t}{\mathscr{G}[\rho_t]} = \twoddx{}\left[D[x,\rho_t]\rho_t(x)\right].
\end{equation*}

\Cref{thm:varEquivCDLor} tells us that the associated Lorenz variable dynamics, \begin{equation*}
    \ddt{\LL_t} = -\widetilde{D}[f,\LL_t]\bigg/\twoddf{\LL_t},
\end{equation*} are in fact given by \begin{equation*}
    \ddt{\LL_t} = +\grad{\LLfrak,\cdspace,\LL_t}{\widetilde{\mathscr{G}}[\LL_t]}.
\end{equation*}

Pleasantly, when the Gini coefficient is defined over the space of Lorenz curves, $\widetilde{\mathscr{G}},$ it has a rather simple geometric interpretation in terms of the area contained by the Lorenz curve. Whereas when the Gini coefficient is defined over $\{\rho\in\pplusac{\posreal}\,:\,\int_{\posreal}\,dx\,x\rho(x)=1\},$ the functional is proportional to an expected absolute difference of i.i.d. copies of the distribution and is less immediately interpretable as an economic inequality metric. In this way, the variational formulation in the Lorenz dynamics is somewhat more intuitive.

\Cref{thm:varEquivCDLor} establishes that the variational principle of maximally increasing economic inequality for these dynamics, previously only known in the $(x,t,\rho)$ variables from \cite{CDSpaces2024}, is now equally valid in the more economically-meaningful $(f,t,\LL)$ variables.

\subsection{The global isometry}

\begin{remark}[The $\cdspace$ global metric, \cite{CDSpaces2024}]
    Consider two positive distributions $\hat{\rho}_0$ and $\hat{\rho}_1$ over $\real$ with the same first moment. The $\cdspace$ metric associated to $\langle\,,\,\rangle_{\rho,\cdspace}$, \cref{def:w22Dtensor}, is \begin{equation}\label{eq:cdspaceglobalmetric}
         \begin{split}
            d_{\cdspace}^2(\hat{\rho}_0,\hat{\rho}_1) =\inf_{\text{curves } \rho_t} \bigg\{&\int_0^1\,dt\,\int_\real\,dx\,D[x,\rho_t] \left(\twoddx{\psi_t} \right)^2 \text{where } \ddt{\rho_t} = \twoddx{}\left(D[x,\rho_t]\twoddx{\psi_t}\right)\,\\&:\, \rho_0=\hat{\rho}_0,\, \rho_1=\hat{\rho}_1, \, \totdd{}{t}\left(\int_\real\,dx\,x\rho_t(x)\right)=0\bigg\}.
        \end{split}
    \end{equation}
\end{remark}

\begin{definition}[$\lorenzsub{\cdspace}$ global metric]\label{def:cdspacelorenzglobalmetric}
Let $||\cdot||_{\LLfrak,\cdspace,\LL}$ be the norm induced by \cref{def:tensorLorCD}. Then between two Lorenz curves $\hat{L}_0$ and $\hat{L}_1$ with $\hat{L}_0(1) = \hat{L}_1(1)$, 
    \begin{align}
        d^2_{\LLfrak,\cdspace}(\hat{L}_0,\hat{L}_1) &= \inf_{\text{curves } \LL_t} \left\{\int_0^1\,dt\,||\dot{\LL}_t||^2_{\LLfrak,\cdspace,\LL_t} \, :\, \LL_0 = \hat{L}_0, \, \LL_1 = \hat{L}_1\right\}\nonumber\\
        &=\inf_{\text{curves } \LL_t} \bigg\{\int_0^1\,dt\,\int_0^1\,df\,\frac{1}{\widetilde{d}[f,\LL_t]}\left(\ddf{\LL_t}\right)^2\left(\ddf{\dot{\LL}_t(f)}\right)^2 \,\nonumber \\
        &\qquad\qquad\qquad:\, \LL_0 = \hat{L}_0, \, \LL_1 = \hat{L}_1, \totdd{\LL_t(1)}{t}=0\bigg\}.\label{eq:LorCDGlobalMetricv2}
    \end{align}
\end{definition}

\begin{theorem}[$\cdspace$ and $\lorenzsub{\cdspace}$ isometry]\label{thm:CDLorIsom}
    For two positive distributions $\hat{\rho}_0$ and $\hat{\rho}_1$ over $\real$ with the same first moment, \[d_{\cdspace}^2(\hat{\rho}_0,\hat{\rho}_1) = d_{\LLfrak,\cdspace}^2(\Lbf[\hat{\rho}_0],\Lbf[\hat{\rho}_1]).\]
\end{theorem}

\begin{proof}
Fix a curve $\rho_t:[0,1]\rightarrow \pplusac{\real}$ such that \begin{equation*}
    \totdd{}{t}\int_\real\,dx\,x\rho_t(x) = 0
\end{equation*} for $t\in[0,1].$ Let $\psi_t$ solve \begin{equation}
    \ddt{\rho_t} = \twoddx{}\left[D[x,\rho_t]\twoddx{\psi_t}\right]\label{eq:cdIsomProofPsit}
\end{equation} and $\LL_t = \Lbf[\rho_t].$

Recall from the proof of \cref{thm:w2MLorIsom} that \begin{equation}
    \ddf{}\dot{\LL}_t = -\frac{1}{\left(\rho_t\circ G[\rho_t]\right)(f)}\int_{-\infty}^{G[\rho_t](f)}\,du\,\ddt{\rho_t(u)}.\nonumber
\end{equation}

In which case, using \cref{eq:cdIsomProofPsit}, \begin{equation}\label{eq:expressftDerivL}
    \ddf{}\dot{\LL}_t = -\frac{1}{\left(\rho_t\circ G[\rho_t]\right)(f)}\ddx{}\left[D[x,\rho_t]\twoddx{\psi_t}\right]_{x=G[\rho_t](f)}.
\end{equation}

Observe from \cref{eq:expressftDerivL} that \begin{equation}
    -\frac{1}{D[x,\rho_t]}\int_{-\infty}^x\,dy\, \rho_t(y) \left[\ddf{}\dot{\LL}_t(f)\right]_{f=C[\rho_t](y)}=\twoddx{\psi_t} \label{eq:cdIsomProof1},
\end{equation} which follows directly from the fundamental theorem of calculus. \Cref{eq:cdIsomProof1} can be simplified by making use of how the spatial differentials transform, \begin{equation}
    \rho\ddf{} = \ddx{}\nonumber,
\end{equation} to obtain the (surprisingly) simple expression 
\begin{equation}
    \twoddx{\psi_t} = -\frac{\dot{\LL}_t(f)|_{f=C[\rho_t](x)}}{D[x,\rho_t]}.\label{eq:cdIsomProof2}
\end{equation}

The expression for $\twoddx{\psi_t}$ given in \cref{eq:cdIsomProof2} may be inserted into the spatial integral in the $\cdspace$ global metric, \cref{eq:cdspaceglobalmetric}, so that \begin{align}
    \int_\real\,dx\,D[x,\rho_t] \left(\twoddx{\psi_t} \right)^2&= \int_\real\,dx\,D[x,\rho_t] \left(\frac{\dot{\LL}_t(f)|_{f=C[\rho_t](x)}}{D[x,\rho_t]} \right)^2 \nonumber\\
    &=\int_\real\,dx\,\frac{1}{D[x,\rho_t]} \left(\dot{\LL}_t(f)|_{f=C[\rho_t](x)} \right)^2\nonumber\\
    &=\int_0^1\,df\,\frac{1}{\widetilde{d}[f,\LL_t]\left[\left(\rho_t\circ G[\rho_t]\right)(f)\right]^2}\left(\dot{\LL}_t(f)\right)^2\nonumber\\
    &= \int_0^1\,df\,\frac{1}{\widetilde{d}[f,\LL_t]}\left(\twoddf{\LL_t(f)}\dot{\LL}_t(f)\right)^2,\label{eq:cdIsomProof3}
\end{align} making use of the integral transformation rule and that $\frac{1}{\left(\rho_t\circ G[\rho_t]\right)(f)} = \twoddf{\LL_t(f)}.$

We conclude by integrating \cref{eq:cdIsomProof3} in $t$ over $[0,1]$ and minimizing over all admissible curves $\rho_t$. Noting the correspondence between such curves and curves $\LL_t$, in the space of Lorenz curves with the same right boundary value yields the equality between the metrics.
\end{proof}

\begin{remark}[The lens of econophysics]
    \Cref{def:cdspacelorenzglobalmetric} presents an entirely novel metric on the space of Lorenz curves, which may be viewed in the following light: The 2-Wasserstein metric measures the distance between two probability measures based upon a kinetic premise about the work involved in moving mass over a distance. Here, there is a kinetic premise about the manner in which idealized economic agents move throughout the range of possible wealths, and this premise of economic mobility is proven, via isometry, to be reflected in the metric over Lorenz curves.\footnote{In particular, the kinetic premise is encoded in the diffusion coefficient, $D[x,\rho],$ which gives the expected infinitesimal variance of an agent with wealth $x$ in the diffusion approximation of a mean-field model.} 
    
    What then is a sensible way to measure the difference between the Lorenz curves of two distributions of wealth? If there is an established assumption about the underlying economic dynamics generating the distributions, then this kinetic notion ought to be taken into account in the sense of \cref{def:cdspacelorenzglobalmetric}, which is connected to \cref{eq:cdspaceglobalmetric} by \cref{thm:CDLorIsom}.
\end{remark}

\appendix

\section{Key results from \texorpdfstring{\cite{LorDyn2024}}{}}\label{sec:reviewOfResultsLorDyn}
We quickly review key results from \cite{LorDyn2024} that are used throughout this paper. The review of those previous results will be done using the notation of the present paper.

Let $\rho_t(x)$ be a smooth time-parameterized curve in $\pplusac{\real}$. The cumulative distribution function (cdf) of the density $\rho_t(x)$ is $C[\rho_t](x) = \int_{-\infty}^x\,dy\,\rho_t(y)$ and its inverse is $G[\rho_t](f)$ for $f\in[0,1]$. For all $x\in\real,$ $\left(G[\rho_t]\circ C[\rho_t]\right)(x) = x.$

An integral over $x\in\real$ can be transformed into an integral over the cdf variable $f\in[0,1]$ by way of the variable transformation $f:=C[\rho_t](x)$ such that \begin{equation*}
    df = \rho_t(x) dx \quad \text{and} \quad dx = \frac{1}{\left(\rho_t\circ G[\rho_t]\right)(f)}df.
\end{equation*}

In this manner integrals transform as \begin{equation}
    \int_\real\,dx\,\rho_t(x)Q(x,t) = \int_0^1\,df\,Q(G[\rho_t](f),t)
\end{equation} and 

\begin{equation}
    \int_0^1\,df\, R(f,t) = \int_\real\,dx\,\rho_t(x)R(C[\rho_t](x),t).
\end{equation}

Moreover differentials transform as \begin{equation*}
    \ddx{} = \left(\rho_t\circ G[\rho_t]\right)(f)\ddf{} \quad \text{and} \quad \ddf{} = \frac{1}{\rho_t(x)}\ddx{}.
\end{equation*}

These relations will be repeatedly used through the derivations in the main body of the paper.

The Lorenz curve of a positive probability density $\rho \in \pplusac{\real}$ is defined as $\LL(f) = (L\circ G[\rho])(f)$ where $L(x) = \int_{-\infty}^x\,dy\,y\rho(y).$ We will denote this by $\Lbf[\rho] = \LL$, and the density associated to a Lorenz curve is $\Pbf[\LL].$

Let $\LL$ be a Lorenz curve and $\rho,$ the associated density. The spatial derivative of a Lorenz curve in the cdf variable \begin{equation}\label{eq:fDerivL}
    \ddf{\LL(f)} = G[\rho](f).
\end{equation} This is easily seen by transforming the differential as \begin{align*}
    \ddf{\LL(f)} &= \left[\frac{1}{\rho(x)}\ddx{}L(x)\right]_{x=G[\rho](f)}\\
    &= x|_{G[\rho](f)}.
\end{align*}

Informally this indicates that the spatial variable $x$ may be replaced by $\totdd{\LL}{f}$, with the appropriately transformed arguments.

The second spatial derivative of a Lorenz curve is the reciprocal of the density, \begin{equation}\label{eq:ffDerivL}
    \tottwodd{\LL(f)}{f} \equiv \totdd{G[\rho](f)}{f}=\frac{1}{\left(\rho\circ G[\rho]\right)(f)}.
\end{equation} This follows from the inverse function theorem, the positivity of $\rho$, and the previous result on $\totdd{\LL}{f}.$ This implies the formal replacement of $\rho$ with $\left(\tottwodd{\LL}{f}\right)^{-1}$.

The final result, which is used in the proofs of the isometry theorems, is the computation of the time-derivative of the inverse cdf $G[\rho_t](f)$. The expression of the time-derivative, \begin{equation}\label{eq:tDerivG}
    \ddt{G[\rho_t](f)} = -\frac{1}{\left(\rho_t\circ G[\rho_t]\right)(f)}\int_{-\infty}^{G[\rho_t](f)}\,dy\, \ddt{\rho_t(y)},
\end{equation} follows from differentiating the identity, \begin{equation*}
    \left(C[\rho_t]\circ G[\rho_t]\right)(f) = f,
\end{equation*} in time and re-arranging to obtain \begin{equation*}
    \ddt{G}=-\ddt{C}\bigg/\dd{C}{G}.
\end{equation*} A more heuristic argument for the same result is found in the first proof of \cref{apdx:techLemmasProofs}.

\section{Proof of technical lemmas}\label{apdx:techLemmasProofs}
\begin{proof}
    Let $\rho_t$ be a time-parameterized, smooth curve in $\pplusac{\real}$ and define $\ddt{}\big|_{t=0}\rho_t(x) = h(x)$. We seek to find \begin{equation*}
        \ddt{}\bigg|_{t=0}G[\rho_t](f),
    \end{equation*} where $G[\rho_t]$ is the inverse cdf of $\rho_t.$
    Recall that $C[\rho_t]$ is the cdf of $\rho_t$ and note that \begin{equation}\label{eq:inverseCdfEquality}
        C[\rho_t]\left(G[\rho_t](f)\right) = C[\rho_0]\left(G[\rho_0](f)\right).
    \end{equation}
    Write $G[\rho_t](f) = G[\rho_0](f) + \left(G[\rho_t](f)-G[\rho_0](f)\right)$ then the left hand side of \cref{eq:inverseCdfEquality} is
    \begin{equation}\label{eq:firstOrderGExpanV1}
        C[\rho_0]\bigg( G[\rho_0](f) + \left(G[\rho_t](f)-G[\rho_0](f)\right)\bigg) + t\int_{-\infty}^{ G[\rho_0](f) + \left(G[\rho_t](f)-G[\rho_0](f)\right)}\,dy\,h(y),
    \end{equation} neglecting second-order and higher terms. The first term of the previous expression can also be expanded to first-order so that \cref{eq:firstOrderGExpanV1} becomes \begin{equation}
        C[\rho_0](G[\rho_0](f))+\left(G[\rho_t](f)-G[\rho_0](f)\right)\rho_0\left(G[\rho_0](f)\right)+t\int_{-\infty}^{ G[\rho_0](f)}\,dy\,h(y).
    \end{equation} 

    Returning to \cref{eq:inverseCdfEquality} and re-arranging gives \begin{equation*}
        \frac{G[\rho_t](f)-G[\rho_0](f)}{t} = -\frac{\int_{-\infty}^{G[\rho_0](f)}\,dy\,h(y)}{\rho_0\left(G[\rho_0](f)\right)} + \text{second-order and higher terms}.
    \end{equation*}
\end{proof}

\begin{proof}
Fix $f\in[0,1]$ and consider a smooth curve of positive probability densities over $\real.$

This result follows from applying the Leibniz integral rule and inserting the previous result.
\begin{align*}
    \ddt{}\bigg|_{t=0}\Lbf[\rho_t](f) &= \ddt{}\bigg|_{t=0}\int_{-\infty}^{G[\rho_t](f)}\,dy\,y\rho_t(y)\\
    &= \int_{-\infty}^{G[\rho_0](f)}\,dy\,yh(y) + G[\rho_0](f)\rho_0\left(G[\rho_0](f)\right) \ddt{}\bigg|_{t=0}G[\rho_t](f)\\
    &=\int_{-\infty}^{G[\rho_0](f)}\,dy\,yh(y) - G[\rho_0](f)\int_{-\infty}^{G[\rho_0](f)}\,dy\,h(y).
\end{align*}
\end{proof}

\bibliographystyle{amsplain}
\bibliography{biblio}

\end{document}